\documentclass[a4paper,11pt]{amsart}
\pdfoutput=1
\usepackage[utf8x]{inputenc}
\usepackage[T1]{fontenc}
\usepackage[english]{babel}
\usepackage[top=1.7cm, bottom=1.7cm, left=1.5cm,right=1.5cm]{geometry}
\usepackage{setspace}
\usepackage{amsmath}
\usepackage{amsthm}
\usepackage{amsfonts}
\usepackage{amssymb}
\usepackage{bm} 
\usepackage{stackengine}
\usepackage{scalerel}
\usepackage{xcolor}
\usepackage{setspace} 
\usepackage{lscape} 

\usepackage{enumitem}
\setlist[enumerate]{itemsep=0mm}
\setlist[itemize]{itemsep=0mm}

\usepackage{mathtools} 

\usepackage{graphicx}
\usepackage{float}
\usepackage[bookmarksnumbered = true,pdfstartview=FitH,linkcolor=blue,citecolor=blue]{hyperref}     
\hypersetup{colorlinks=true,breaklinks=true,urlcolor=black,linkcolor=blue}  

\usepackage{tikz}
\usetikzlibrary{decorations.markings}

\usepackage{etex}
\usetikzlibrary{matrix,arrows,calc,shapes,decorations.pathreplacing}

\newtheorem{thm-intro}{Theorem}
\newtheorem{prop}{Proposition}[section]

\newtheorem{lemma}[prop]{Lemma}

\newtheorem{thm}[prop]{Théorème}
\newtheorem{theorem}[prop]{Theorem}
\theoremstyle{definition}

\newtheorem{definition}[prop]{Definition}

\newtheorem{prop-def}[prop]{Proposition and Definition}

\newtheorem{thm-def}[prop]{Theorem and Definition}

\newtheorem{example}[prop]{Example}

\newtheorem{rmk}[prop]{Remark}

\renewcommand{\Im}{\mathrm{Im\,}}

\newcommand{\Isom}{\mathrm{Isom\,}}
\newcommand{\Pic}{\mathrm{Pic\,}}

\DeclareMathOperator{\Id}{Id}
\newcommand{\N}{\mathbb N}
\newcommand{\Z}{\mathbb Z}
\newcommand{\Q}{\mathbb Q}
\newcommand{\R}{\mathbb R}
\newcommand{\C}{\mathbb C} 
\renewcommand{\hat}{\widehat}
\renewcommand{\tilde}{\widetilde}
\renewcommand{\emptyset}{\varnothing}

\renewcommand{\phi}{\varphi}
\renewcommand{\bar}{\overline}
\renewcommand{\P}{\mathbb P}

\newcommand{\Aut}{\mathrm{Aut\,}}

\renewcommand{\O}{\mathrm{O}}

\newcommand{\intl}{[\![}
\newcommand{\intr}{]\!]}

\setitemize{font=\small, label=$\bullet$}

\renewcommand{\phi}{\varphi}

\begin{document}
\title[Real structures on KLT Calabi-Yau pairs]{Finiteness of real structures on\\ KLT Calabi-Yau regular smooth pairs of dimension 2}
\author{Mohamed Benzerga}
\address{Universit\'e d'Angers, \textsc{Larema}, UMR CNRS 6093, 2, boulevard Lavoisier, 49045 Angers cedex 01, France}
\email{mohamed.benzerga@univ-angers.fr}
\noindent

\bibliographystyle{amsalpha} 
\begin{abstract}
In this article, we prove that a smooth projective complex surface $X$ which is regular (i.e. such that $h^1(X,\mathcal O_X)=0$) and which has a $\R$-divisor $\Delta$ such that $(X,\Delta)$ is a KLT Calabi-Yau pair has finitely many real forms up to isomorphism. For this purpose, we construct a complete CAT(0) metric space on which $\Aut X$ acts properly discontinuously and cocompactly by isometries, using Totaro's Cone Theorem. Then we give an example of a smooth rational surface with finitely many real forms but having a so large automorphism group that \cite{mon-papier} does not predict this finiteness.
\end{abstract}
\renewcommand{\subjclassname}
{\textup{2010} Mathematics Subject Classification}
 
\maketitle

\section*{Introduction}
\indent A \textit{real form} of a complex projective variety $X$ is a scheme over $\R$ whose complexification is $\C$-isomorphic to $X$. A \textit{real structure} on $X$ is an antiregular (or antiholomorphic) involution $\sigma : X\to X$ (cf. \cite[Chap. 2]{le-bouquin}). Two real structures $\sigma$ and $\sigma'$ are \textit{equivalent} if there is a $\C$-automorphism $\phi$ of $X$ such that $\sigma'=\phi\sigma\phi^{-1}$.\\
\indent By Weil descent of the base field (cf. \cite[III.\S 1.3]{serre-cohgal-english}), there is a bijective correspondence between the set of $\R$-isomorphism classes of real forms of $X$ and the set of equivalence classes of real structures on $X$. Moreover, if $\sigma$ is a real structure on $X$, this set is parametrized by the \textit{first Galois cohomology set} $H^1(G,\Aut_{\C}X)$, where $G=\langle\sigma\rangle$ acts on $\Aut_{\C}X$ by conjugation.\\
\indent The results of this paper are motivated by the study of the finiteness problem for real forms of rational surfaces. We already addressed this question in our previous article \cite{mon-papier} whose main result, combined with [\textit{loc. cit.} \S 3.2], states as follows:
\begin{thm-intro}\label{first-thm} Let $X$ be a smooth complex rational surfaces and let $\Aut^*X$ be the image of the natural morphism $\Aut X\to \O(\Pic X)$.\\
If $\Aut^*X$ does not contain a non-abelian free group $\Z*\Z$ then $X$ has finitely many real forms up to $\R$-isomorphism.\end{thm-intro}
However, this result does not completely solve the problem since there are rational surfaces whose automorphism group does contain a non-abelian free group (cf. Example \ref{example-12-points}). In fact, it is not known how $\Aut^*X$ can be large for a rational surface $X$. For example, up to our knowledge, the problem of the finite generation of the group $\Aut^*X$ is open (but Lesieutre constructed in \cite{lesieutre} a six-dimensional variety $X$ such that $\Aut^*X$ is not finitely generated and he showed that $X$ is an example of a smooth projective variety having infinitely many non-isomorphic real forms\footnote{To the best of our knowledge, it is the first known example.}).\\
\indent The aim of this article is to prove the following result:
\begin{thm-intro}\label{main-thm} Let $X$ be a smooth projective complex surface which is regular (i.e. $q(X) := h^1(X,\mathcal O_X) = 0$).\\
If there is a $\R$-divisor $\Delta$ on $X$ such that $(X,\Delta)$ is a KLT Calabi-Yau pair\footnote{See Definition \ref{def-klt-cy}}, then $X$ has finitely many real forms up to $\R$-isomorphism.\end{thm-intro}
The proof of this result uses different kind of tools than those we used in \cite{mon-papier} and mostly geometric actions on complete CAT(0) metric spaces: roughly speaking, a metric space is CAT(0) if it has "nonpositive curvature". We will give precise definitions in section \ref{section-cat0}. Then, we recall the definition of KLT Calabi-Yau pairs and we prove finiteness of real forms for them using Totaro's Cone Theorem \ref{cone-thm}. We give an example of a rational surface whose finiteness of real forms cannot be deduced from Theorem \ref{first-thm} but is obtained from Theorem \ref{thm-finiteness-klt}. Finally, we present an example of a rational surface for which the finiteness problem remains open and which can be equipped of a $\Q$-divisor $\Delta$ such that $(X,\Delta)$ is log-canonical Calabi-Yau. 

\section{Preliminaries: Geometric actions on CAT(0) spaces} \label{section-cat0}
We begin this section by a brief explanation of the link between finiteness of real forms and geometric (i.e. proper and cocompact) actions on CAT(0) spaces, which are a generalization of manifolds with nonpositive curvature (see \cite[I.1.3, II.1.1]{bridson-haefliger}): this will be our main tool in order to turn our finiteness problem into a problem of hyperbolic geometry.

\begin{definition} \label{def-cat0}
\begin{itemize}
\item A \textbf{geodesic} between two points $a$ and $b$ in a metric space $(X,d)$ is a map $\gamma : [0,\ell]\to X$ such that $\gamma(0)=a$, $\gamma(\ell)=b$ and $\forall t,t'\in[0,\ell],\;d(\gamma(t),\gamma(t'))=|t-t'|$ (in particular, $\gamma$ is continuous and $\ell = d(a,b)$). A \textbf{geodesic triangle} $\Delta$ in $X$ consists of three points $x,y,z\in X$ and three geodesic segments $[x, y], [y, z], [z, x]$.
\item A metric space $(X,d)$ is \textbf{geodesic} if every two points in $X$ are joined by a geodesic (not necessarily unique).
\item A geodesic metric space $(X,d)$ is said to be a $\mathbf{\operatorname{\textbf{CAT}}(0)}$ \textbf{space} if for every geodesic triangle $\Delta$ in $X$, there exists a triangle $\Delta'$ in $\R^n$ endowed with the euclidean metric, with sides of the same length as the sides of $\Delta$, such that distances between points on $\Delta$ are less than or equal to the distances between corresponding points on $\Delta'$. \index{CAT(0) space}
\item \emph{\cite[I.8.2]{bridson-haefliger}\footnote{As explained in \textit{op.cit.}, I.8.3, if every closed ball of $X$ is compact, then this definition is equivalent to the standard definition where the open balls are replaced by the compact subsets of $X$: for us, this is always the case.}} Let $\Gamma$ be a group acting by isometries on a metric space $X$. The action is said to be \textbf{proper} (or \textbf{properly discontinuous}) if\footnote{Denoting by $B(x,r) = \{ y\in X | d(x,y) <r\}$ the open ball of center $x\in X$ and radius $r\geq 0$.}
$$\forall x\in X,\; \exists r > 0 ,\;\{\gamma\in\Gamma|\gamma .B(x, r) \cap B(x, r) \neq \emptyset\} \text{ is finite}.$$
\end{itemize}
\end{definition}
\begin{center}
\includegraphics[scale=0.35]{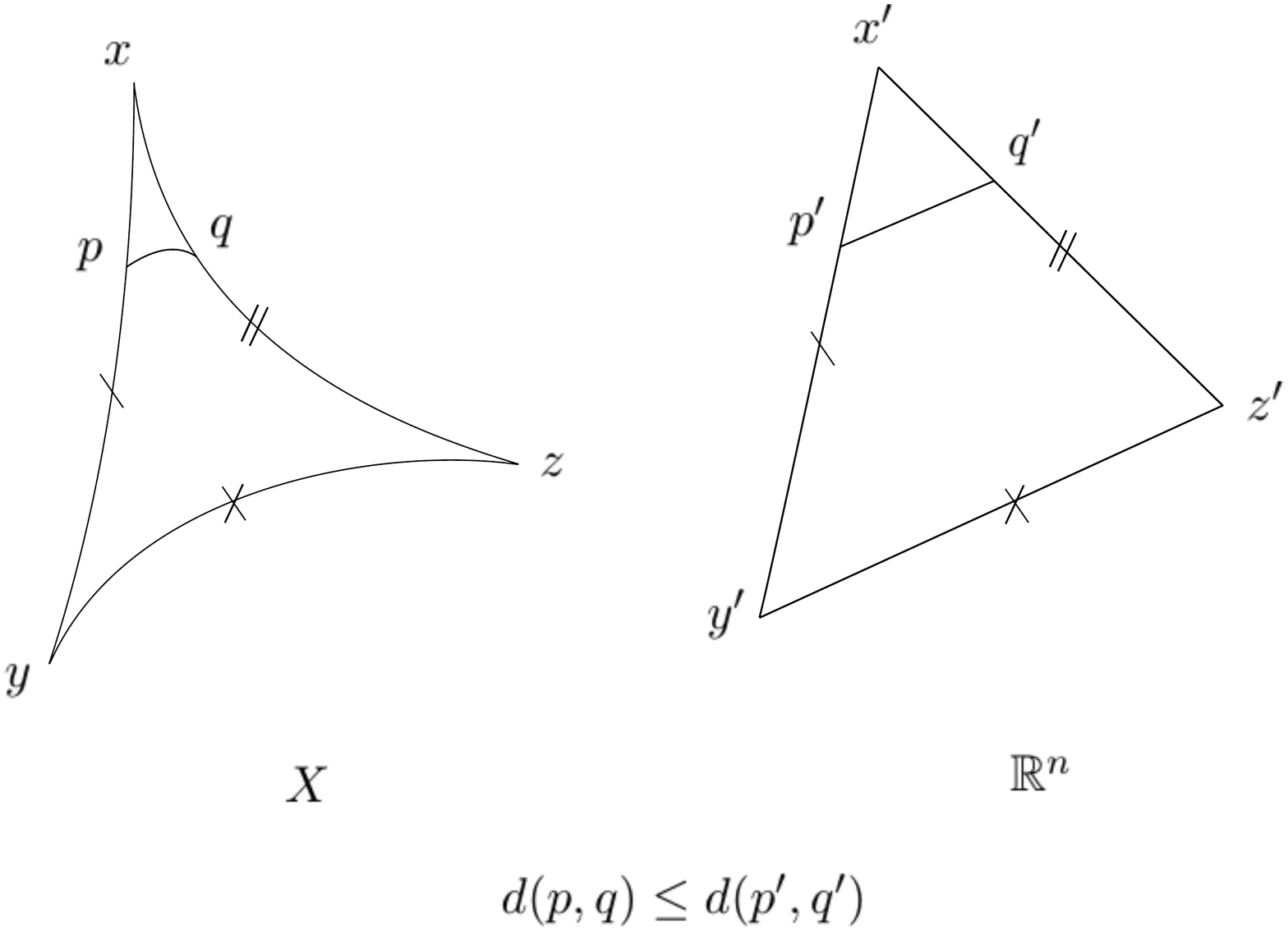}
\end{center}

\begin{theorem} {\bf (\cite[II.2.8]{bridson-haefliger})} \label{conj-classes-cat0}\\ 
If a group $\Gamma$ acts geometrically (i.e. properly discontinuously and cocompactly by isometries) on a complete CAT(0) space, then $\Gamma$ contains only finitely many conjugacy classes of finite subgroups.\qed
\end{theorem}

\section{Finiteness theorem}

Firstly, let us introduce the surfaces we deal with (cf. \cite{totaro}, \cite[8.2,8.4]{totaro2} for a slightly more general definition):
\begin{definition} {\bf (KLT Calabi-Yau pair)} \label{def-klt-cy} \index{KLT Calabi-Yau pair}\\
Let $X$ be a \textit{smooth} projective complex variety and $\Delta$ be a $\R$-divisor on $X$.\\
$(X,\Delta)$ is \textbf{a KLT (\textit{resp.} log-canonical) Calabi-Yau pair} if there exists a resolution $\pi : (\tilde{X},\tilde{\Delta})\to (X,\Delta)$ satisfying the following conditions:
\begin{itemize}
\item $K_{\tilde{X}}+\tilde{\Delta} = \pi^*(K_X+\Delta)$;
\item $\tilde{\Delta}$ has simple normal crossings and his coefficients are $< 1$ (KLT condition), \textit{resp.} $\leq 1$ (log-canonical condition);
\item $\Delta$ is an effective $\R$-divisor such that $K_X+\Delta$ is numerically trivial (Calabi-Yau condition).
\end{itemize}
\end{definition}

\begin{example} \label{example-klt} Let us present here some examples of KLT Calabi-Yau pairs. The reader may also look at Examples \ref{example-12-points} and \ref{exemple-dolgachev}.
\begin{itemize}
\item Of course, there are irrational surfaces $X$ having a $\R$-divisor $\Delta$ such that $(X,\Delta)$ is KLT Calabi-Yau (simply think of $X$ being Calabi-Yau smooth and $\Delta=0$). There are less trivial examples, like some $\P^1$-bundles over elliptic curves (cf. \cite[1.4]{alexeev-mori}).
\item If $X$ is a \textit{Halphen surface} of index $m\geq 2$ (for the definition, cf. \cite[\S 2]{cantat-dolgachev}) and $F$ a reduced fibre of the elliptic fibration on $X$ with simple normal crossings, then $\left(X,\dfrac{1}{m}F \right)$ is a KLT Calabi-Yau pair: for, the definition of a Halphen surface shows that $F\sim -mK_X$ and $\dfrac{1}{m}<1$. If $X$ is of index 1, then $\left(X,\dfrac{1}{2}(F+F')\right)$ is a KLT Calabi-Yau pair for two distinct reduced smooth fibers $F$ and $F'$ of $X\to\P^1$.
\item Similarly, if $X$ is a \textit{Coble surface} and if the special fibre $F$ (see \cite[Prop. 3.1]{cantat-dolgachev}) is reduced and has simple normal crossings, then $\left(X,\dfrac{1}{2}F \right)$ is a KLT Calabi-Yau pair.
\end{itemize}
\end{example}

For our purposes, we need the following finiteness theorem:
\begin{theorem} {\bf (Cone theorem for KLT Calabi-Yau pairs - \cite[8.7]{totaro2})}\label{cone-thm}\index{KLT Calabi-Yau! Cone theorem for}\index{Cone theorem}\\
Let $(X,\Delta)$ be a KLT Calabi-Yau pair.\\
If $X$ is a surface, then the action of $\Aut X$ on the nef cone has a rational polyhedral fundamental domain (i.e. it is the closed convex cone spanned by a \emph{finite} set of Cartier divisors in $\Pic X\otimes_{\Z}\R$).\qed
\end{theorem} 
%

The aim of this article is to prove the following result:

\begin{theorem} \label{thm-finiteness-klt} 
Let $X$ be a smooth projective complex surface which is regular (i.e. $q(X) := h^1(X,\mathcal O_X) = 0$).\\
If there is a $\R$-divisor $\Delta$ on $X$ such that $(X,\Delta)$ is a KLT Calabi-Yau pair, then $X$ has finitely many real forms up to $\R$-isomorphism.\end{theorem}

Thus, Example \ref{example-klt} shows that our previous result about finiteness of real forms for Cremona special surfaces, i.e.  (cf. \cite[3.2]{mon-papier}) is a special case of this result when the fibre $F$ we mentioned in \ref{example-klt} above is reduced with simple normal crossings. \vspace{.5cm}\\
\noindent \textbf{Strategy of the proof. }--- Let $\sigma$ be a real structure on $X$ and let $\Aut^{\#}X$ (resp. $\Aut^*X$) be the kernel (resp. the image) of the natural morphism $p : \Aut X\to \O(\Pic X)$. If $G=\langle\sigma\rangle$ acts on $\Aut X$ by conjugation (i.e. $\forall\phi\in\Aut X,\;\sigma.\phi := \sigma\phi\sigma^{-1}$), then the exact sequence
$$1\longrightarrow \Aut^{\#}X \longrightarrow \Aut X\longrightarrow \Aut^*X \longrightarrow 1$$
is $G$-equivariant and induces an exact sequence in Galois cohomology. By \cite[I.\S 5.5, Cor. 3]{serre-cohgal-english}, it suffices to prove that $H^1(G,\Aut^*X)$ is finite and that $\forall b\in Z^1(G,\Aut X),\;H^1(G,(\Aut^{\#}X)_b)$ is finite. But this last condition is true for every smooth irreducible projective complex variety by \cite[III.\S 4.3]{serre-cohgal-english} and \cite[1.2]{mon-papier} (see also the proof of Theorem 2.5 in \textit{loc. cit.}). Thus, we need only to show the finiteness of $H^1(G,\Aut^*X)$.\\
\indent Now, for the special case of KLT Calabi-Yau pairs, the idea is the following: using Totaro's Cone Theorem, we will construct a complete CAT(0) space on which $\Aut^*X\rtimes \langle\sigma^*\rangle$ acts geometrically (where $\sigma^*\in\O(\Pic X)$ is the isometry induced by $\sigma$). Then, we will be able to conclude that $H^1(G,\Aut^*X)$ is finite using Theorem \ref{conj-classes-cat0} together with the following result (which can be proved easily using the definitions as in the proof of \cite[Th. 2.4]{mon-papier}) :\vspace{-.1cm}
\begin{lemma}\label{finiteness-conj-classes}
Let $G=\langle \sigma\rangle \simeq \Z/2\Z$, $A$ be a $G$-group and $A\rtimes G$ the semidirect product defined by the action of $G$ on $A$.\\
If $A\rtimes G$ has a finite number of conjugacy classes of elements of order 2 (in particular, if it has finitely many conjugacy classes of finite subgroups), then $H^1(G,A)$ is finite.
\end{lemma}
\begin{center} --- $***$ --- \end{center}
\vspace{.5cm}
Before beginning the proof of this theorem, let us give some definitions to clarify the terms we use:
\begin{definition} \label{def-poly-funda-domain} Let $X$ be either $\mathcal H^n$ or $\R^n$ (\footnote{In what follows, when writing $\R^n$, it is understood as $\R^n$ equipped with its euclidean metric (which is denoted by $E^n$ in \cite{ratcliffe}).}).
\begin{itemize}
\item A subset $C$ of $X$ is \textbf{convex} if $\forall x,y\in C$, the geodesic segment linking $x$ and $y$ is contained in $C$.
\item A \textbf{side} of $C$ is a maximal nonempty convex subset of the relative boundary $\partial C$ (cf.\cite[p. 195, 198]{ratcliffe}).
\item A \textbf{(convex) polyhedron} of $X$ is a nonempty closed convex subset of $X$ whose collection of sides is locally finite. In what follows, we will always say "polyhedron" instead of "convex polyhedron".
\item Let $X$ be a subset of either $\mathcal H^n$ or $\R^n$. A \textbf{fundamental polyhedron (or polyhedral fundamental domain)} for the action of a discrete group $\Gamma$ of isometries of $X$ is a polyhedron $P$ whose interior $\mathring{P}$ is such that the elements of $\{g(\mathring{P}),\;g\in\Gamma\}$ are pairwise disjoint and $\displaystyle X=\bigcup_{g\in\Gamma} g(P)$. Moreover, $P$ is a \textbf{locally finite fundamental polyhedron} if the set $\{g(P),\;g\in\Gamma\}$ is locally finite, i.e. if for all compact $K\subseteq X$, there are only finitely many elements of $\{g(P),\;g\in\Gamma\}$ which intersect $K$.
\end{itemize}
\end{definition}

\begin{proof}[Proof of Theorem \ref{thm-finiteness-klt}]\label{beginning-proof-klt}
We begin by explaining how we can turn our problem into a problem of hyperbolic geometry. Hodge index Theorem shows that the signature of the intersection form on $\mathrm{NS}(X) = \Pic X$ is $(1,n)$, where $\mathrm{rk}\,\Pic X = n+1$. Note that this is the only place where we use the fact that $h^1(X,\mathcal O_X)=0$. In fact, we could try to remove this hypothesis but we should replace $\Aut^{\#}X$ and $\Aut^*X$ with the analogous groups corresponding to the action of $\Aut X$ on $\mathrm{NS}(X)$ instead of $\Pic X$ but we do not have a general result of cohomological finiteness for the kernel of the action of $\Aut X$ on $\mathrm{NS}(X)$ whereas we gave such a result for the kernel of the action of $\Aut X$ on $\Pic X$ in the paragraph "Strategy of the proof" above.\\
\indent Thus, we obtain the \textit{hyperboloid model} of the hyperbolic space $\mathcal H^n := \{v \in \Pic X\otimes_{\Z}\R\,|\;v^2=1,\;v.H>0\}$ equipped with the distance $d :(u,v) \mapsto \mathrm{argcosh}(u.v)$ (where $u.v$ is the intersection product of $u$ and $v$ and $H$ is an ample divisor class on $X$.\\
\indent The radial projection $\pi : \Pic X\otimes_{\Z}\R \to \Pic X\otimes_{\Z}\R$ from the origin onto the hyperplane $\{v \in \Pic X\otimes_{\Z}\R\,|\;v.E_0=1\}\simeq \R^n$ maps the hyperboloid $\mathcal H^n$ onto the open unit ball $D^n$ of this hyperplane: when endowed with the appropriate metric, this is the \textit{Klein (projective) model} of $\mathcal H^n$ and $\pi$ restricts to an isometry $\mathcal H^n\to D^n$. The geodesic lines of this model are straight line segments so that the convex subsets of $D^n$ (for the hyperbolic metric) are exactly its convex subsets for the euclidean metric of $D^n\subseteq \R^n$. Note that $\pi$ maps the isotropic half-cone $\{v \in \Pic X\otimes_{\Z}\R\,|\;v^2=0,\;v.E_0>0\}$ onto the boundary $\partial D^n$ (which can be seen as the set of lines of this isotropic half-cone). We also want to mention the \textit{Poincaré ball model} $B^n$ of $\mathcal H^n$ which is obtained from the hyperboloid model by means of a stereographic projection from the south pole of the unit sphere of $\Pic X\otimes_{\Z}\R$ on the hyperplane  $\{v \in \Pic X\otimes_{\Z}\R\,|\;v.E_0=0\}$.\\
\indent Finally, if we denote by $\text{Nef}\,X$ the nef effective cone of $X$ and if $N := \pi( \text{Nef}\,X\cap \mathcal H^n) \simeq (\text{Nef}\,X\cap \mathcal H^n)/\R^*$, then we see easily that $N$ is a closed convex subset of $D^n$ and that, by Totaro Cone Theorem \ref{cone-thm}, $\Aut X$ (or $\Aut^*X$) acts on it with a \textit{finitely sided} polyhedral fundamental domain (namely, the projection onto $D^n$ of a polyhedral fundamental domain of the action on $\text{Nef}\,X$). Note that, as we said in the statement of Theorem \ref{cone-thm}, there is a fundamental domain $\mathcal P$ of the action of $\Aut X$ on $\mathrm{Nef}\,X$ which is the closed convex cone generated by finitely many points of $\Pic X\otimes_{\Z}\R$; hence, $\bar{P}:=\pi(\mathcal P)\subseteq \bar{D^n}$ is the convex hull of finitely many points and classical results about convex polyhedra of $\R^n$ show that such a convex set is the intersection of finitely many half spaces and has finitely many sides (which are defined by the bounding hyperplanes of $\bar{P}$). Hence, this is also true for $P := \bar{P}\cap D^n$, the fundamental polyhedron of the action of $\Aut X$ on $N$.\\
\indent In order to use Lemma \ref{finiteness-conj-classes} and Theorem \ref{conj-classes-cat0}, we have to prove that $\Aut^*X\rtimes \langle\sigma^*\rangle$ acts properly and cocompactly by isometries on a CAT(0) complete metric space. In fact, we are reduced to prove Lemma \ref{ratcliffe-modified}, which is the adaptation to our case of \cite[12.4.5, 1$\Rightarrow$ 2]{ratcliffe} (where we replaced a fundamental domain of the action on $\mathcal H^n$ by a fundamental domain on a closed convex subset, which is our $N$), and Lemma \ref{cat(0)-truncated-horoballs}. \textit{The proof ends on page \pageref{end-proof-klt}.} 
\end{proof}

\begin{lemma} \label{proper-action-discrete} Any discrete subgroup $\Gamma$ of $\Isom(\mathcal H^n)$ acts properly discontinuously on $\mathcal H^n$. \end{lemma}

\begin{proof} Note that the action of $\Isom(\mathcal H^n) \simeq \mathrm{O}^+(1,n)$ on $\mathcal H^n$ is transitive and that the stabilizer of a point $x\in\mathcal H^n$ (in the hyperboloid model) is the orthogonal group $\mathrm{O}(x^{\perp}) \simeq \O_n(\R)$. Thus, this action induces a bijection $\mathcal H^n \simeq \mathrm{O}^+(1,n)/\O_n(\R)$. Since $\Gamma$ is discrete in the locally compact group $\mathrm{O}^+(1,n)$ and since $\O_n(\R)$ is compact, the result follows from \cite[3.1.1]{wolf}. \end{proof}


Before stating our lemmas \ref{ratcliffe-modified} and \ref{cat(0)-truncated-horoballs}, let us give some other definitions (cf. \cite{ratcliffe}):
\begin{definition} Let $\Gamma$ be a discrete subgroup of $\Isom(\mathcal H^n)$.
\begin{itemize} 
\item A point $a\in\partial \mathcal H^n$ is a \textbf{limit point} of $\Gamma$ if there is a point $x$ of $\mathcal H^n$ and a sequence $(g_i)$ of elements of $\Gamma$ such that $(g_i(x))$ converges to $a$.
\item In the Poincaré ball model $B^n$, an \textbf{horoball} based at a point $a\in\partial B^n$ is an Euclidean ball contained in $\bar{B^n}$ which is tangent to $\partial B^n$ at the point $a$.
\item Assume $\Gamma$ contains a parabolic element (cf. \cite[\S 4.7]{ratcliffe}) having $a\in\partial \mathcal H^n$ as its fixed point.\\
A \textbf{horocusp region} is an open horoball $B$ based at a point $a\in\partial \mathcal H^n$ such that $$\forall g\in\Gamma\setminus \text{Stab}_{\Gamma}(a),\;g(B)\cap B = \emptyset.$$
\end{itemize}
\end{definition}
The following Lemma develops and makes more precise an idea of Totaro in \cite[\S 7]{totaro2}: 
\begin{lemma}\label{ratcliffe-modified} Let $\Gamma$ be a discrete subgroup of $\Isom(\mathcal H^n)$, $L(\Gamma)$ the set of limit points of $\Gamma$ in $\bar{\mathcal H^n}$, $C(\Gamma)$ the convex hull of $L(\Gamma)$ in $\bar{\mathcal H^n}$ and $N$ a $\Gamma$-invariant closed convex subset of $\mathcal H^n$.\\
If the action of $\Gamma$ on $N$ has a finitely sided polyhedral fundamental domain $P$, then there exists a finite union (maybe empty) $V_0$ of horocusp regions with disjoint closures such that $(P\cap C(\Gamma))\setminus V_0$ is compact.
\end{lemma}
\begin{proof}[Sketch of proof]
The proof of this lemma is an adaptation of the proof of \cite[12.4.5, 1$\Rightarrow$ 2]{ratcliffe}: we replaced the fundamental domain of the action on $\mathcal H^n$ by a fundamental domain of the action on a closed convex subset, which is our $N$ and we replaced geometrical finiteness hypothesis for $\Gamma$ (which is more general than the existence of a finitely sided fundamental polyhedral domain of $\Gamma$ on $\mathcal H^n$) by the hypothesis of the existence of a finitely sided fundamental polyhedral domain for the action of $\Gamma$ on $N$. Thus, we have to check all the proofs of the results used by \cite{ratcliffe} in the proof of \cite[12.4.5, 1$\Rightarrow$ 2]{ratcliffe} in order to replace $\mathcal H^n$ by a closed convex subset $N$. The details of these verifications are in the \hyperref[appendix-section]{Appendix}. Here, we sum up the main points: 
\begin{itemize}[label=$\bullet$,font=\small]
\item if $P_0$ is a fundamental polyhedron of the action of $\Gamma$ on $\mathcal H^n$, then $P=P_0\cap N$ is a fundamental polyhedron of the action of $\Gamma$ on $\mathcal H^n$
\item by (6.6.10, 8.5.7), like $\mathcal H^n$, a closed convex subset $N$ of $\mathcal H^n$ is a proper geodesically connected and geodesically complete metric space\footnote{A metric space is \textbf{proper} or \textbf{finitely compact} if every bounded closed subset of it is compact.}: indeed, $N$ is proper as a subspace of the proper metric space $\mathcal H^n$, $N$ is geodesically complete, as it is complete, and it is geodesically connected, since it is convex. 
From this fact, we can deduce that the action of $\Gamma$ on $N$ has a (locally finite) exact convex fundamental polyhedron, e.g. a Dirichlet polyhedron, by (5.3.5, 6.6.13) and (6.7.4 (2)) since the group is discrete (and hence acts properly discontinuously, cf. Lemma \ref{proper-action-discrete}) and since there is a point $a\in N$ whose stabilizer $\Gamma_a$ is trivial (by (6.6.12)).
\item for the other points, it is a question of replacing $\mathcal H^n$ by $N$ and checking that everything remains true (sometimes by using convexity and/or closedness of $N$ in $\mathcal H^n$).
\end{itemize}

\end{proof} 
 
\begin{lemma} \label{cat(0)-truncated-horoballs}
Let $C$ be a closed convex subset of $\mathcal H^n$, $\Gamma$ a discrete subgroup of $\Isom(\mathcal H^n)$ stabilizing $C$, $V_0$ a finite family of open horoballs with disjoint closures and $\displaystyle V_1 := \bigcup_{\gamma\in\Gamma} \gamma(V_0)$.\\
There is a family of open horoballs with disjoint closures, obtained by shrinking the horoballs of $V_1$, whose union $U$ is such that $C\setminus U$ is a complete CAT(0) space.
\end{lemma}

\begin{proof} 
By \cite[II.11.27]{bridson-haefliger}, for every family $U$ of \textit{disjoint}\footnote{This is the key point which explains why we had to prove Lemma \ref{ratcliffe-modified}.} open horoballs, $\mathcal H^n\setminus U$ is a complete CAT(0) space for the induced length metric (this distance is defined between 2 points as the infimum of the lengths of rectifiable curves of $\mathcal H^n\setminus U$ between those two points ; it is different from the metric induced by the hyperbolic metric on $\mathcal H^n\setminus U$). Thus, $C\setminus U$ is complete as a closed subset of the complete space $\mathcal H^n\setminus U$. \\
It remains to study geodesic connectedness (term of \cite[\S 1.4]{ratcliffe}) or convexity (term of \cite[I.1.3]{bridson-haefliger}) of $C\setminus V_1$ in $\mathcal H^n\setminus V_1$ to conclude (using \cite[II.1.15.(1)]{bridson-haefliger}) that $C\setminus V_1$ is CAT(0) for the metric induced by the distance of $\mathcal H^n\setminus V_1$ (which is itself the length metric induced by the metric of $\mathcal H^n$). So let $x,y\in C\setminus V_1$ : if the geodesic $\gamma$ of $\mathcal H^n$ joining $x$ and $y$ is contained in $\mathcal H^n\setminus V_1$, then it is also contained in $C\setminus V_1$ since $C$ is convex. Otherwise,  $\Im\gamma$ passes through at least one horoball and \cite[II.11.33, II.11.34]{bridson-haefliger} shows that a geodesic $\delta$ of $\mathcal H^n\setminus V_1$ linking $x$ and $y$ is obtained by concatenation of the hyperbolic geodesics which are tangent to the bounding horospheres of the horoballs crossed by $\gamma$ on the one hand and geodesics of these horo\textit{spheres} on the other hand.
\textit{A priori}, it may happen that $\Im\delta$ is not contained in $C$. But we can shrink $V_0$ so that the antipodal point of the base point of each horosphere of $V_0$ belongs to $C$ (\footnote{More simply, the horospheres do not "get out" of $C$.}), which causes this effect on all the horoballs of $V_1$ under a finite number of operations.
Hence, the geodesics of the horospheres are contained in $C$ and the hyperbolic geodesics are contained in $C$ because they join two points of $C$. Thus, if we denote by $U$ the result of this shrinking of $V_1$, we showed that $C\setminus U$ is geodesically connected in $\mathcal H^n\setminus U$ and this shows that $C\setminus U$ is a complete CAT(0) space.
\end{proof}

\begin{proof}[End of the proof of Theorem \ref{thm-finiteness-klt}] \label{end-proof-klt}
\indent We apply Lemma \ref{ratcliffe-modified} with $\Gamma = \Aut^*X$, $N = \pi( \text{Nef}\,(X)\cap \mathcal H^n) \simeq \text{Nef}\,(X)/\R^*$ and $P$ being a fundamental polyhedron of the action of $\Gamma$ on $N$: this gives us a finite family $V_0$ of open horoballs \textit{with disjoint closures} and a convex subset $P_C = P\cap C(\Gamma)$ of $P$ such that $P_C\setminus V_0$ is \textit{compact}. Thus, $\Gamma$ acts properly (by Lemma \ref{proper-action-discrete}) and cocompactly on $C\setminus V_1$ where $C := N\cap C(\Gamma)$ and $\displaystyle V_1 := \bigcup_{\gamma\in\Gamma} \gamma(V_0)$ (note that $C(\Gamma)$ is a $\Gamma$-invariant closed convex subset of $\mathcal H^n$).\\
\indent Now, by Lemma \ref{cat(0)-truncated-horoballs}, we can replace $V_1$ by another family $U$ of open horoballs \textit{with disjoint closures} such that $C\setminus U$ is a complete CAT(0) space. The compacity of a fundamental domain is preserved by this shrinking because $P\setminus (U\cap P)$ is a bounded closed subset of $\mathcal H^n$, so it is compact because $\mathcal H^n$ is a proper metric space. By the way, one can verify that the proof of \cite[12.4.5, 1$\Rightarrow$ 2]{ratcliffe} allows to shrink the horoballs without trouble.\\
\indent \textit{Finally,} we can conclude that $\Aut^*X$ acts properly and cocompactly on the complete CAT(0) space $C\setminus U$. It is not enough: in order to apply Lemma \ref{finiteness-conj-classes} and Theorem \ref{conj-classes-cat0}, we must obtain the same result for $\Aut^*X\rtimes \langle\sigma^*\rangle$, where $\sigma$ is a real structure on $X$.
By Lemma \ref{proper-action-discrete}, the discrete isometry group $\Aut^*X\rtimes \langle\sigma^*\rangle$ acts properly on $\mathcal H^n$ (hence also on $\mathcal H^n\cap (C\setminus U)$). Since there is a fundamental domain of $\Aut^*X\rtimes \langle\sigma^*\rangle$ which is a closed subset of that of $\Aut^*X$ which is compact, we see that $\Aut^*X\rtimes \langle\sigma^*\rangle$ also acts cocompactly: this concludes the proof.
\end{proof} 

\begin{rmk} In fact, we could give a much shorter proof of Theorem \ref{thm-finiteness-klt} if $N$ were smooth complete.\\
\indent First, note that $N$ is a pinched Hadamard manifold\footnote{A \textbf{pinched Hadamard manifold} is a complete simply connected Riemannian manifold whose all sectional curvatures lie between two negative constants.} as a convex subset of $\mathcal H^n$: in particular, note that it is 
simply connected because of its convexity (which can be seen in the Klein model, where convexity is the same as Euclidean convexity and really implies simply connectedness).\\
\indent Now, if we denote by $P$ a fundamental domain of the action of $\Aut X$ on $N\subseteq D^n$, we remark that $\overline{P}$ is a fundamental domain of the action of $\Aut X$ on $\bar{N}\subseteq\overline{D^n}$ and that it is a convex polyhedron of the Klein model $\overline{D^n}$ of $\overline{\mathcal H^n}$ and also of the Euclidean space $\R^n$ (since convexity in the Klein model is the same as Euclidean convexity). Indeed, by Definition \ref{def-poly-funda-domain}, we need to check that $\bar{P}$ has a locally finite collection of sides. But this is true since Totaro Cone Theorem \ref{cone-thm} shows that it is finitely sided (as we have seen in the beginning of the proof of Theorem \ref{thm-finiteness-klt}, see page \pageref{beginning-proof-klt}). Thus, by \cite[6.4.8]{ratcliffe}, $\bar{P}$ has finite volume. Therefore, $P$ is also of finite volume. Since there is a fundamental domain $P'$ of $\Aut X\rtimes \langle\sigma\rangle$ which is contained in $P$, we see that $P'$ is also of finite volume. Thus, the quotient of $N$ by $\Gamma:= \Aut X\rtimes \langle \sigma\rangle$ is a finite volume quotient of the pinched Hadamard manifold $N$ and \cite[5.4.2, F5, 6.1, 5.5.2]{bowditch95} 
shows that $\Gamma$ has finitely many conjugacy classes of finite subgroups: thus $X$ has a finite number of real forms by Lemma \ref{finiteness-conj-classes}.\\

\end{rmk}

\section{Two examples} 

\begin{example} \label{example-12-points} Here we study the example given by Totaro in \cite{totaro}: it is a blow-up $X$ of $\P^2$ at 12 points and we show that $\Aut X$ contains a subgroup isomorphic to $\Z*\Z$. Since there exists a $\R$-divisor $\Delta$ such that $(X,\Delta)$ is a KLT Calabi-Yau pair, $X$ has finitely many non-isomorphic real forms and this finiteness cannot be deduced from Theorem \ref{first-thm}.\\
\indent Let $\zeta=e^{2i\pi/3}$. We denote by $X$ the blow-up of $\P^2$ at the 12 points of the set $\mathcal P = \{[1:\zeta^i:\zeta^j],\;(i,j)\in\intl 0;2\intr^2\}\cup\{[1:0:0],\;[0:1:0],\;[0:0:1]\}$. Let $C_1,\dots,C_9$ be the lines of $\P^2$ of equations $(y=x),\;(y=\zeta x),\;(y=\zeta^2 x),\;(z=x),\;(z=\zeta x),\;(z=\zeta^2 x),\;(z=y),\;(z=\zeta y),\;(z=\zeta^2 y)$. We can easily verify that each line $C_i$ passes exactly through 4 of the points of $\mathcal P$ and that each point blown-up is the intersection point of exactly 3 of the $C_i$: this is called the \textit{dual of Hesse configuration}.\footnote{Hesse configuration itself is not interesting for our purposes: since it contains exactly 9 points (and 12 lines), the surface obtained by blowing up these points has finitely many non-equivalent real structures by Theorem \ref{first-thm}.}\\
\indent Note that $\displaystyle \left(X,\frac{1}{3}\sum_{i=1}^9 \hat{C_i}\right)$ is a KLT Calabi-Yau pair because:
\begin{itemize}
\item $X$ is smooth;
\item $\displaystyle \frac{1}{3}\sum_{i=1}^9 \hat{C_i}$ has simple normal crossings and its coefficients are < 1;
\item $\displaystyle -K_X=\frac{1}{3}\sum_{i=1}^9 \hat{C_i}$.
\end{itemize}
\indent Now, let us give some results about $\Aut X$. Firstly, one can show that if a line $D$ passes through one of the points $[1:0:0]$, $[0:1:0]$, $[0:0:1]$ and one of the $[1:\zeta^i:\zeta^j]$, then $D$ is one of the $C_i$: for example, if $D$ passes through $[1:0:0]$ and one of the $[1:\zeta^i:\zeta^j]$, then, in the affine chart $(x\neq 0)$ of $\P^2$, we have $D=(z=\zeta^{j-i}y)$.\\
\indent We claim that \fbox{$\Aut^{\#} X=\{\Id\}$}: since all the points blown-up belong to $\P^2$, it suffices to check that there does not exist any line passing through at least 11 of the 12 points of $\mathcal P$ (in fact, $\Aut^{\#}X$ is non-trivial if and only if $\mathcal P$ does not contain 4 points in general position, i.e. if and only if all the points of $\mathcal P$ are collinear except maybe only one of them). But if $D$ was such a line, then it would necessarily pass through one of the points $[1:0:0]$, $[0:1:0]$, $[0:0:1]$ and one of the $[1:\zeta^i:\zeta^j]$: thus, $D$ would be one of the $C_i$. Since none of the $C_i$ passes through 11 of the 12 points blown-up, we see that $D$ does not exist: this proves the claim.\\
\indent By the end of the example of \cite[\S 2]{totaro}, we have (denoting $E:=\C/\Z[\zeta]$):
\begin{center}
\fbox{$\displaystyle \Aut X \simeq \Aut^* X = (\Z/3\Z)^2\rtimes \frac{GL_2(\Z[\zeta])}{\Z/3\Z} = \Aut ((E\times E)/(\Z/3\Z))$}
\end{center}
\indent Finally, we want to show that $\Aut X$ contains a subgroup isomorphic to $\Z*\Z$:
\begin{itemize}[label=$\bullet$,font=\small]
\item it is well-known that $SL_2(\Z)$ contains finite index subgroups isomorphic to $\Z*\Z$ (for example, $\displaystyle S:=\left\langle\begin{bmatrix}
1 & 2\\ & 1
\end{bmatrix},\begin{bmatrix}
1 & \\2 & 1
\end{bmatrix}\right \rangle$ has index 12, cf. \cite[II.25]{de-la-harpe}) ;
\item $GL_2(\Z[\zeta])$ acts on $E\times E = \C^2/(\Z[\zeta]^2)$ by matrix product and $\displaystyle \frac{GL_2(\Z[\zeta])}{\Z/3\Z}$ is the quotient by the subgroup generated by $\zeta.I_2$. Clearly, two elements of $SL_2(\Z)$ (or even $GL_2(\Z)$) are never equal modulo $\langle \zeta.I_2\rangle$ so $SL_2(\Z)$ injects into $\displaystyle \frac{GL_2(\Z[\zeta])}{\Z/3\Z}$: this concludes the proof.
\end{itemize}
\end{example}

\begin{example} {\bf (\cite[6.10]{dolgachev-zhang})} \label{exemple-dolgachev} As promised, we now describe the example of a rational surface for which the finiteness problem for real forms remains open. \\
Let $L_1,\dots,L_5$ be five lines in general linear position in $\P^2$. For $i,j\in\intl 1;4\intr$, we denote by $p_{ij}$ the intersection point of $L_i$ and $L_j$. Let us fix a cubic $C_3$ passing through the points $p_{ij}$ and intersecting $L_5$ at three distinct points $q_1$, $q_2$, $q_3$. Finally, let $a$ be another point of $L_5$.\\
We consider the blow-up $X$ of $\P^2$ at the 10 points $p_{ij}$, $q_k$ and $a$: it is a nodal Coble surface since $|-K_X|=\varnothing$ and $|-2K_X| = \{C_6:= R_1+R_2+R_3+R_4+2R_5\}$, where the $R_i$'s are the strict transform of the $L_i$'s in $X$ (note that $X$ is nodal since $R_1$,\dots,$R_4$ are $(-2)$-curves).\\
In \cite{dolgachev-zhang}, it is claimed that $\Aut X$ has infinitely many orbits on the set of (-1)-curves of $X$ but, in a private communication, Dolgachev explained me that there is a gap in the proof of this fact (more precisely, the elements of the group $G$ constructed in \textit{op. cit.} cannot be lifted to the double covering $S(A)$ of $X$ ramified along $R_1+\dots+R_4$). Note that if it were true, this would show that $X$ does not contain a divisor $\Delta$ such that $(X,\Delta)$ is a KLT Calabi-Yau pair. For, if such a divisor existed, then Cone Theorem would imply that $\Aut X$ has finitely many orbits on the extremal rays of the nef cone of $X$ and this would be true also for its dual cone, which is the cone of curves of $X$ (cf. \cite[4.1]{looijenga}): this is absurd because $(-1)$-curves form an $\Aut X$-invariant subset of the set of extremal rays of $\bar{NE}(X)$.\\
However, note that $\left(X,\dfrac{1}{2}C_6\right)$ is a log-canonical Calabi-Yau pair since $\dfrac{1}{2}C_6 = \dfrac{1}{2}(R_1+R_2+R_3+R_4)+R_5$ has clearly simple normal crossings, has coefficients $\leq 1$ and satisfies the condition $K_X+\dfrac{1}{2}C_6 \equiv 0$.\\
\end{example}


\section{Appendix} \label{appendix-section}
In this appendix, we provide a detailed proof of Lemma \ref{ratcliffe-modified}, i.e. a detailed inspection of all the proofs of the results used by \cite{ratcliffe} in the proof of \cite[12.4.5]{ratcliffe} in order to replace $\mathcal H^n$ by a closed convex subset $N$. We recall here the statement of Lemma \ref{ratcliffe-modified}: \vspace{0.6cm}\\
\textbf{Lemma \ref{ratcliffe-modified}.} \textit{Let $\Gamma$ be a discrete subgroup of $\Isom(\mathcal H^n)$, $L(\Gamma)$ the set of limit points of $\Gamma$ in $\bar{\mathcal H^n}$, $C(\Gamma)$ the convex hull of $L(\Gamma)$ in $\bar{\mathcal H^n}$ and $N$ a $\Gamma$-invariant closed convex subset of $\mathcal H^n$.\\
If the action of $\Gamma$ on $N$ has a finitely sided polyhedral fundamental domain $P$, then there exists a finite union (maybe empty) $V_0$ of horocusp regions with disjoint closures such that $(P\cap C(\Gamma))\setminus V_0$ is compact.}
\vspace{.6cm}

In what follows, all numbers like (12.4.2) refer to \cite{ratcliffe}. Moreover, when some notations are undefined, please consider they are the same as in \cite{ratcliffe}, \emph{mutatis mutandis}. Finally, when some results cited in the diagram page \pageref{diagramme-horoboules} are not cited in the text below, then these are general results which apply to our case either without any change, or changing only $\mathcal H^n$ into $N$.
\\

Some remarks are widely used below so we gather them here:
\begin{itemize}
\item if $P_0$ is a fundamental polyhedron of the action of $\Gamma$ on $\mathcal H^n$, then $P=P_0\cap N$ is a fundamental polyhedron of the action of $\Gamma$ on $\mathcal H^n$
\item by (6.6.10, 8.5.7), like $\mathcal H^n$, a closed convex subset $N$ of $\mathcal H^n$ is a proper geodesically connected and geodesically complete metric space\footnote{A metric space is \textbf{proper} or \textbf{finitely compact} if every bounded closed subset of it is compact.}: indeed, $N$ is proper as a subspace of the proper metric space $\mathcal H^n$, $N$ is geodesically complete, as it is complete, and it is geodesically connected, since it is convex. 
From this fact, we can deduce that the action of $\Gamma$ on $N$ has a (locally finite) exact convex fundamental polyhedron, e.g. a Dirichlet polyhedron, by (5.3.5, 6.6.13) and (6.7.4 (2)) since the group is discrete (and hence acts properly discontinuously, cf. Lemma \ref{proper-action-discrete}) and since there is a point $a\in N$ whose stabilizer $\Gamma_a$ is trivial (by (6.6.12)).
\item if $\Gamma$ is a discrete group of isometries of $\mathcal H^n$ (seen as Poincaré half-space), then the stabilizer $\Gamma_{\infty}$ of the point at infinity induces a discrete subgroup of $\Isom(\R^{n-1}) = \Isom(\partial\mathcal H^n\setminus\{\infty\})$. By (5.4.6), there is a $\Gamma_{\infty}$-invariant affine subspace $Q$ of $\R^{n-1}$ of dimension $m\leq n-1$ and $\Gamma_{\infty}$ is a finite extension of a $\Z^m$. By (7.5.2), $\Gamma_{\infty}$ is a crystallographic isometry group of $\R^m\simeq Q$, i.e. $Q/\Gamma_{\infty}$ is compact
\item for the other points, it is a question of replacing $\mathcal H^n$ by $N$ and checking that everything remains true (sometimes by using convexity and/or closedness of $N$ in $\mathcal H^n$).
\end{itemize}

\textbf{(12.3.7)}: $\Gamma$ is a discrete subgroup of $\Isom(\mathcal H^n)$ so we can define "limit point", "bounded parabolic point"... with regard to its action on the whole space $\mathcal H^n$. Note that $a$ is a limit point if and only if $\exists (g_i)\in\Gamma^{\N}, \forall x\in\mathcal H^n, g_i(x)\xrightarrow[i\to +\infty]{} a$: in particular, if $x\in N$, then $\forall i, g_i(x)\in N$. The rest of the proof can be followed, except that we can check that the geodesic ray $R_i$ is contained in $N$.\\




\textbf{(12.4.3)} Firstly, note that (12.4.2) is not necessary for our purposes since $P$ is finitely-sided. We can make the same reasoning with $N$ instead of $\mathcal H^n$: if $P$ is a fundamental polyhedron of $\Gamma$ acting on $N$, then $\{g(P)|\;g\in\Gamma\}$ is an exact tessellation of $N$ and $\{\nu g(P)|\;g\in\Gamma\} = \mathcal T$ is an exact tessellation of $\nu(N)\subseteq \R^{n-1}$. But $\displaystyle \bigcup_{g\in\Gamma} g(P) = N$ so $U\subseteq \nu(N)$. Since $U$ is an open closed subset of $\R^{n-1}$ and $U\subseteq \nu(N)$,  we see that $U$ is open and closed in the non-empty connected space $\nu(N)$ so that $U=\nu(N)$.\\

\textbf{(12.4.4)} The beginning of the proof remains valid: it shows that if $x\in\bar{P}\cap L(\Gamma)$, where $\bar{P}$ is the closure of $P$ in $\mathcal H^n$, then the stabilizer $\Gamma_x$ is infinite and elementary of parabolic type (cf. \cite[\S 5.5]{ratcliffe}). Of course, $\mathcal T$ is an exact tessellation of $\nu(N)$ instead of $\R^{n-1}$. If $c\in N$ is a cusp point of $\Gamma$, then $U(Q,r)\cap N\neq \emptyset$ because $U(Q,r)$ is a neighborhood of $c$: thus it suffices to replace $U(Q,r)$ by $N\cap U(Q,r)$ in the end of the proof to conclude.\\


\textbf{(12.4.5)} 
Firstly, we note that (12.4 Cor. 3) is a direct corollary of (12.3.7), (12.4.1) and (12.4.4) and that $\bar{P}$ is the closure of $P$ in $\mathcal H^n$.
It suffices to replace:
\begin{itemize}
\setlength\itemsep{.25em}
\item "$\Gamma$ is geometrically finite" by "the fundamental polyhedron of $\Gamma$ on $N$ is finitely-sided" (see (\S 12.4, Example 1));
\item in view of the statement of our Lemma \ref{ratcliffe-modified}, all the statements made in the proof of (12.4.5) concerning $\pi$, $V$, $M$ are useless for our purposes and all we need is $\displaystyle V_0 := \bigcup_{i=1}^m B_i$
\end{itemize}
and we have to note that $K$ is a closed subset of $B^n$ included in the closed subset $N$ (since $P\subseteq N$) hence $K$ is a closed subset of $N$.

\setlength{\fboxsep}{1.5mm}
\begin{landscape}
\begin{center}
\begin{figure}

\begin{tikzpicture}[scale=1.4]
\tikzstyle{fleche}=[->,thick,>=latex]
\node (A) at (0,0) {\fbox{12.4.5 (1$\Rightarrow$2)}};

\setlength{\fboxrule}{2pt}
\node (B) at (-5,1) {\fbox{12.3.5}};
\node (D) at (5,1) {\fbox{12.3.Lemma 2}};
\node (J) at (-6,2) {\fbox{12.3.Lemma 1}};
\setlength{\fboxrule}{0.5pt}
\node (C) at (0,1) {\fbox{12.4.Cor.3}};

\node (F) at (-3.5,2) {\fbox{12.3.Cor.2}};
\node (G) at (-1,2) {\fbox{12.4.4 (1$\Rightarrow$2$\Rightarrow$3)}};
\setlength{\fboxrule}{2pt}
\node (H) at (1,2) {\fbox{12.4.1}};
\setlength{\fboxrule}{0.5pt}
\node (I) at (4.5,2) {\fbox{12.3.7}};

\setlength{\fboxrule}{2pt}
\node (M) at (-6,4) {\fbox{5.4.6}};
\node (N) at (-4.5,4) {\fbox{7.5.2}};
\setlength{\fboxrule}{0.5pt}
\node (O) at (0,4) {\fbox{12.4.3}};
\node (Q) at (-1.5,3) {\fbox{12.3.4}};

\setlength{\fboxrule}{2pt}
\node (R) at (1.5,4) {\fbox{6.4.2}}; 
\node (S) at (3,4) {\fbox{6.2.6}};
\node (U) at (4.5,4) {\fbox{6.4.4}};
\node (V) at (6,4) {\fbox{6.4.5 }};
\node (W) at (8,4) {\fbox{12.3.Lemma 3}};
\node (T) at (9.75,5) {\fbox{6.6.12}};
\node (Y) at (11.25,5) {\fbox{6.6.13}};
\setlength{\fboxrule}{0.5pt}
\node (X) at (10.5,4) {\fbox{6.7.Lemma 1}};


\node (E) at (-1.5,6) {\fbox{6.8.5}};
\setlength{\fboxrule}{2pt}
\node (L) at (2,6) {\fbox{12.4.Lemma 3}};
\node (P) at (2.25,5) {\fbox{6.3.13}};
\setlength{\fboxrule}{0.5pt}
\node (Z) at (-1.5,5) {\fbox{6.8.2}};
\setlength{\fboxrule}{2pt}
\node (A1) at (-3,5) {\fbox{6.4.1}};
\setlength{\fboxrule}{0.5pt}

\setlength{\fboxrule}{2pt}
\node (B1) at (-2,7) {\fbox{6.6.8 }};
\node (C1) at (-0.75,7) {\fbox{6.8.1}};
\setlength{\fboxrule}{0.5pt}
\node (D1) at (-3,6) {\fbox{6.7.5}};
\setlength{\fboxrule}{2pt}
\node (H1) at (-5.5,5) {\fbox{6.6.14}};
\node (K1) at (-5.5,7) {\fbox{6.6.12}};
\node (F1) at (-3,4) {\fbox{6.4.3}};
\node (G1) at (-1.5,4) {\fbox{12.3.3}}; 
\setlength{\fboxrule}{0.5pt}
\node (E1) at (-4.5,6) {\fbox{6.7.4}};

\draw[fleche] (B) -- (A);
\draw[fleche] (C) -- (A);
\draw[fleche] (D) -- (A);
\draw[fleche] (B) -- (F);
\draw[fleche] (J) -- (F);
\draw[fleche] (F) -- (G);
\draw[fleche] (G) -- (C);
\draw[fleche] (H) -- (C);
\draw[fleche] (I) -- (C);

\draw[fleche] (M) -- (G);
\draw[fleche] (N) -- (G);
\draw[fleche] (O) -- (G);
\draw[fleche] (Q) -- (G);

\draw[fleche] (S) -- (I);
\draw[fleche] (R) -- (I);
\draw[fleche] (U) -- (I);
\draw[fleche] (V) -- (I);
\draw[fleche] (W) -- (I);
\draw[fleche] (X) -- (I);

\draw[fleche] (T) -- (X);
\draw[fleche] (Y) -- (X);

\draw[fleche] (I) -- (O);
\draw[fleche] (H) -- (O);
\draw[fleche] (E) -- (O);
\draw[fleche] (L) -- (O);
\draw[fleche] (P) -- (O);
\draw[fleche] (Z) -- (O);
\draw[fleche] (A1) -- (O);
\draw[fleche] (A1) -- (Z);

\draw[fleche] (Z) -- (E);
\draw[fleche] (A1) -- (E);
\draw[fleche] (B1) -- (E);
\draw[fleche] (C1) -- (E);
\draw[fleche] (D1) -- (E);
\draw[fleche] (E1) -- (D1);
\draw[fleche] (F1) -- (Q);
\draw[fleche] (G1) -- (Q);

\draw[fleche] (H1) -- (E1);
\draw[fleche] (K1) -- (E1);

\end{tikzpicture}
\caption{We framed with bold lines the "initial" results, i.e. those whose proof does not require anything else that standard definitions and results (in topology, group theory, etc.) or those whose statement can be adapted to our case without examining their proof and the results used by it.}\label{diagramme-horoboules}
\end{figure}
\end{center}
\end{landscape}

\noindent \textbf{Acknowledgements.} The author is grateful to Frédéric Mangolte for asking him this question, and also for his advice and help. We want to thank Julie Déserti, Igor Dolgachev, Viatcheslav Kharlamov, Stéphane Lamy and Burt Totaro for useful comments, discussions or emails.

\bibliography{Biblio.bib}

\def\cprime{$'$} \def\cprime{$'$}
\providecommand{\bysame}{\leavevmode\hbox to3em{\hrulefill}\thinspace}
\providecommand{\MR}{\relax\ifhmode\unskip\space\fi MR }
\providecommand{\MRhref}[2]{%
  \href{http://www.ams.org/mathscinet-getitem?mr=#1}{#2}
}
\providecommand{\href}[2]{#2}
\begin{thebibliography}{Bow95}

\bibitem[AM04]{alexeev-mori}
Valery Alexeev and Shigefumi Mori, \emph{Bounding singular surfaces of general
  type}, Algebra, arithmetic and geometry with applications ({W}est
  {L}afayette, {IN}, 2000), Springer, Berlin, 2004, pp.~143--174.

\bibitem[Ben16]{mon-papier}
Mohamed Benzerga, \emph{Real structures on rational surfaces and automorphisms
  acting trivially on {P}icard groups}, Math. Z. \textbf{282} (2016), no.~3-4,
  1127--1136.

\bibitem[BH99]{bridson-haefliger}
Martin~R. Bridson and Andr{\'e} Haefliger, \emph{Metric spaces of non-positive
  curvature}, Grundlehren der Mathematischen Wissenschaften [Fundamental
  Principles of Mathematical Sciences], vol. 319, Springer-Verlag, Berlin,
  1999.

\bibitem[Bow95]{bowditch95}
Brian~H. Bowditch, \emph{Geometrical finiteness with variable negative
  curvature}, Duke Math. J. \textbf{77} (1995), no.~1, 229--274.

\bibitem[CD12]{cantat-dolgachev}
Serge Cantat and Igor Dolgachev, \emph{{Rational surfaces with a large group of
  automorphisms}}, J. Amer. Math. Soc. \textbf{25} (2012), no.~3, 863--905.

\bibitem[dlH00]{de-la-harpe}
Pierre de~la Harpe, \emph{Topics in geometric group theory}, Chicago Lectures
  in Mathematics, University of Chicago Press, Chicago, IL, 2000.

\bibitem[DZ01]{dolgachev-zhang}
Igor~V. Dolgachev and De-Qi Zhang, \emph{Coble rational surfaces}, Amer. J.
  Math. \textbf{123} (2001), no.~1, 79--114.

\bibitem[Les17]{lesieutre}
John Lesieutre, \emph{A projective variety with discrete, non-finitely
  generated automorphism group},
  \href{https://arxiv.org/abs/1609.06391}{arXiv:1609.06391}, 2017.

\bibitem[Loo14]{looijenga}
Eduard Looijenga, \emph{Discrete automorphism groups of convex cones of finite
  type}, Compos. Math. \textbf{150} (2014), no.~11, 1939--1962.

\bibitem[Man]{le-bouquin}
Fr\'ed\'eric Mangolte, \emph{Vari\'et\'es alg\'ebriques r\'eelles}, \textit{to
  appear}, cf. \url{http://math.univ-angers.fr/\~mangolte/}.

\bibitem[Rat06]{ratcliffe}
John~G. Ratcliffe, \emph{Foundations of hyperbolic manifolds}, second ed.,
  Graduate Texts in Mathematics, vol. 149, Springer, New York, 2006.

\bibitem[Ser02]{serre-cohgal-english}
Jean-Pierre Serre, \emph{Galois cohomology}, english ed., Springer Monogr.
  Math., Springer-Verlag, Berlin, 2002, Translated from the French by Patrick
  Ion and revised by the author.

\bibitem[Tot10]{totaro}
Burt Totaro, \emph{The cone conjecture for {C}alabi-{Y}au pairs in dimension
  2}, Duke Math. J. \textbf{154} (2010), no.~2, 241--263.

\bibitem[Tot12]{totaro2}
\bysame, \emph{Algebraic surfaces and hyperbolic geometry}, Current
  developments in algebraic geometry, Math. Sci. Res. Inst. Publ., vol.~59,
  Cambridge Univ. Press, Cambridge, 2012, pp.~405--426.

\bibitem[Wol84]{wolf}
Joseph~A. Wolf, \emph{Spaces of constant curvature}, Publish or Perish. Fifth
  edition, 1984.

\end{thebibliography}
\end{document}